\documentclass[12pt]{article}%
\usepackage{amsmath}
\usepackage{amsfonts}
\usepackage{amssymb}
\usepackage{accents}
\usepackage{graphicx}%
\providecommand{\U}[1]{\protect\rule{.1in}{.1in}}
\newtheorem{theorem}{Theorem}[section]

\newtheorem{lemma}[theorem]{Lemma}

\newtheorem{proposition}[theorem]{Proposition}

\newenvironment{proof}[1][Proof]{\noindent\textbf{#1.} }{\ \rule{0.5em}{0.5em}}
\begin{document}

\title{The Spectra of Arrangement Graphs}
\author{Jos\'{e} Araujo\\Facultad de Ciencias Exactas, UNICEN \\Tandil, Argentina.
\and Tim Bratten\\Facultad de Ciencias Exactas, UNICEN\\Tandil, Argentina. }
\date{}
\maketitle

\begin{abstract}
Arrangement graphs were introduced for their connection to computational
networks and have since generated considerable interest in the literature. In
a pair of recent articles by Chen, Ghorbani and Wong, the eigenvalues for the
adjacency matrix of an (n,k)-arrangement graph are studied and shown to be
integers. In this manuscript, we consider the adjaceny matrix directly in
terms of the representation theory for the symmetric group. Our point of view
yields a simple proof for an explicit fomula of the associated spectrum in
terms of the characters of irreducibile representations evaluated on a
transposition. As an application we prove a conjecture raised by Chen,
Ghorbani and Wong.

\end{abstract}

\section{Introduction}

Since arrangement graphs were first introduced in the seminal paper
\cite{day}, there has been considerable interest in the literature. In two
recent articles, \cite{chen1} and \cite{chen2}, the eigenvalues of the
adjacency matrix are studied. The main result in \cite{chen2} is that the
eigenvalues are integers. In this article we study the adjacency matrix from
the perspective of the representation theory of symmetric groups. In
particular, we consider the representation associated to the arrangement graph
and the corresponding equivariant operator associated to the adjacency matrix.
Our approach leads to a simple derivation for an explicit formula of the
spectra of $(n,k)$-arrangement graphs in terms of the characters of
irreducible representations evaluated on a transposition.

It seems pertinent to point out that the problem of translating our formula,
constructed in representational theoretic language, into a combinatorial
algorithm to compute the eigenvalues, has already been solved. In particular,
the input into this algorithm is a finite sequence $\lambda=(\lambda
_{1},\ldots,\lambda_{j})$ \ of positive integers, called \emph{a partition of}
$k$, that satisfies
\[
\lambda_{1}\geq\lambda_{2}\geq\cdots\geq\lambda_{j}\text{ \ and \ }\lambda
_{1}+\lambda_{2}+\cdots+\lambda_{j}=k.
\]
A formula due to Frobenius directly calculates the contribution to the
eigenvalue associated to the corresponding irreducible representation of the
symmetric group $S_{k}$. The second part of the algorithm constructs a
partition $\mu$ of $n$ via a combinatorial rule on the Young diagram
$Y(\lambda)$ called Pieri's formula. Once again Frobenius's formula calculates
the contribution to the eigenvalue, this time associated to the corresponding
irreducible representation of the symmetric group $S_{n}$. The one other
component of the formula is the binomial coefficient
\[
\binom{n-k}{2}.
\]

As an application of our formula, we prove a conjecture by Chen, Ghorbani and
Wong that states for $k$ fixed and $n$ large, $-k$ is the only negative
eigenvalue in the spectrum of the $\left(  n,k\right)  $-arrangement graph.

\section{The representation associated to an arrangement graph}

Suppose $\Gamma$ is a graph and $G$ is a subgroup of the automorphism group of
$\Gamma$. Let $V$ be the set of vertices of $\Gamma$ and let $\mathbb{V}$ be
the complex vector space with basis vectors from $V$. The group $G$ acts
naturally on $\mathbb{V}$. We refer to this as \emph{the associated }%
$G$-\emph{representation}. Let $U$ denote the adjacency matrix of $\Gamma$.
The matrix $U$ is directly related to the linear operator $\Upsilon$ on
$\mathbb{V}$ defined by
\[
\Upsilon\left(  v\right)  =\sum_{w\sim v}w
\]
where $v$ and $w$ are vertices and $w\sim v$ if they share an edge. In
particular, the matrix for $\Upsilon$ in a basis of vertices is $U$. Observe
that
\[
\Upsilon\left(  g\cdot v\right)  =g\cdot \Upsilon\left(  v\right)  \text{ \ for
each }v\in\mathbb{V}.
\]
We call $\Upsilon$ \emph{the equivariant operator associated to} $U$.

We now introduce the graphs to be considered. For each natural number $n$, let
$I_{n}=\left\{  1,2,\ldots,n\right\}  $. It will be convenient for us to
identify the symmetric group $S_{n}$ with the permutation group of $I_{n}$.
Suppose $k$ is a natural number and $k\leq n$. By definition, \emph{a}
$k$-\emph{permutation} $\sigma$ is an injective map
\[
\sigma:I_{k}\rightarrow I_{n}\text{.}%
\]
Let $V(n,k)$ be the set of $k$-permutations. \emph{The arrangement graph}
$A(n,k)$ is the graph whose vertices are the elements of $V(n,k)$, where two
$k$-permutations are adjoined by an edge if they agree as functions on exactly
$k-1$ elements of $I_{k}$. The symmetric group $G=S_{n}$, acts naturally on
$A(n,k)$ by the formula
\[
\pi\cdot\sigma=\pi\circ\sigma
\]
where $\pi\in S_{n}$ and $\sigma\in V(n,k)$. Note that when $k=n$ then
$A(n,k)$ is an edgeless graph and the adjacency matrix $U=0$, so we may assume
$k<n$.

Let $\mathbb{V}$ be the associated $S_{n}$-representation. Our first task is
to show $\mathbb{V}$ is induced from a representation of a subgroup. In
particular, the symmetric group $S_{k}$, of $I_{k}$, is naturally identified
with a subgroup of $S_{n}$. We also identify the permutations $S(I_{n}-I_{k}%
)$, of the set $I_{n}-I_{k}$, with a subgroup of $S_{n}$ . The subgroup $H$ of
$S_{n}$ generated by these two subgroups is naturally isomorphic to the direct
product
\[
H=S_{k}\times S\left(  I_{n}-I_{k}\right)  .
\]
We will also identify $S_{k}$ with a subset of the vertices: $S_{k}\subseteq
V(n,k)$. If $\mathbb{W}$ is the subspace of $\mathbb{V}$ with basis vectors
from $S_{k}$ then $\mathbb{W}$ is an $H$ invariant subspace. In particular,
$\mathbb{W}$ is the regular representation of $S_{k}$ with a trivial action of
$S\left(  I_{n}-I_{k}\right)  $.

To show that the representation of $G=S_{n}$ in $\mathbb{V}$ is induced from
the representation of $H$ in $\mathbb{W}$ we use the following lemma. The
proof is an exercise in the concept of induced representations and is left to
the reader.

\begin{lemma}
Suppose $\mathbb{V}$ is a representation of a finite group $G$. Let
$H\subseteq G$ be a subgroup and $\mathbb{W}\subseteq\mathbb{V}$ an
$H$-invariant subspace. Suppose that:\newline\ \ \ \ \ (1) The $G$-module
generated by $\mathbb{W}$ is $\mathbb{V}$ and \newline\ \ \ \ \ (2) For each
$g\in G-H$ we have $g\cdot W\cap W=\left\{  0\right\}  $ \newline Then there
is a natural isomorphism
\[
\mathbb{V}\cong\text{Ind}_{H}^{G}(\mathbb{W}).
\]

\end{lemma}

\begin{proposition}
Suppose $G=S_{n}$ and $\mathbb{V}$ is the representation associated to the
arrangement graph. If $\mathbb{W}$ and $H$ are defined as above then
\[
\mathbb{V}\cong\text{Ind}_{H}^{G}(\mathbb{W}).
\]

\end{proposition}

\begin{proof}
If $\pi\notin H$ then $\pi(j)\in I_{n}-I_{k}$ for some $j\in I_{k}$. Hence,
for each basis vector $\sigma\in S_{k}$ it follows $\pi\cdot\sigma\notin
S_{k}$ and therefore
\[
\pi\cdot\mathbb{W}\cap\mathbb{W}=\left\{  0\right\}  \text{.}%
\]
Next observe that each basis vector in $V(n,k)$ can be written as $\pi
\cdot\sigma$ for some $\sigma\in S_{k}$. Therefore $G$-module generated by
$\mathbb{W}$ is $\mathbb{V}$ and the result follows from the previous lemma.
\end{proof}

The irreducible representations of the symmetric group are parametrized by
partitions. In particular, \emph{a partition} $\lambda$ \emph{of a positive
integer} $m$ is a finite sequence of positive integers $\lambda=(\lambda
_{1},\ldots,\lambda_{j})$ such that
\[
\lambda_{1}\geq\lambda_{2}\geq\cdots\geq\lambda_{j}\text{ \ and \ }\lambda
_{1}+\lambda_{2}+\cdots+\lambda_{j}=m
\]
If $\lambda$ is a partition of $m$ we write $\lambda\vdash m$ let $S^{\lambda
}$ denote a realization of the corresponding irreducible representation for
$S_{m}$. If $\lambda\vdash k$ we let
\[
\mathbb{W}_{\lambda}\subseteq\mathbb{W}%
\]
denote the corresponding $S_{k}$-isotypic component in $\mathbb{W}$. Since
$\mathbb{W}$ is the regular representation for $S_{k}$ it follows that
\[
\mathbb{W}_{\lambda}\cong m_{\lambda}S^{\lambda}%
\]
where $m_{\lambda}$ is the dimension of $S^{\lambda}$.

\begin{proposition}%
\[
\mathbb{V}=%
{\displaystyle\bigoplus\limits_{\lambda\vdash k}}
\text{Ind}_{H}^{G}(\mathbb{W}_{\lambda})
\]

\end{proposition}

\begin{proof}
The result is clear since
\[
\mathbb{V}=\text{Ind}_{H}^{G}(\mathbb{W})=\text{Ind}_{H}^{G}(%
{\displaystyle\bigoplus\limits_{\lambda\vdash k}}
\mathbb{W}_{\lambda})\cong%
{\displaystyle\bigoplus\limits_{\lambda\vdash k}}
\text{Ind}_{H}^{G}(\mathbb{W}_{\lambda}).
\]

\end{proof}

\section{The equivariant operator associated to the adjacency matrix}

We need to introduce some notation. If $p$ is any integer, it will be
convenient to define
\[
\binom{p}{2}=\frac{p(p-1)}{2}.
\]
Suppose $m$ is a positive integer and let $i,j$ be positive integers with
$i<j\leq m$. We let $(ij)$ denote the corresponding transposition in $S_{m}$.
Let $\mathbb{C}$ denote the complex numbers. We define an element of the group
algebra $T_{m}\in\mathbb{C}\left[  S_{m}\right]  $ according to the formula
\[
T_{m}=%
{\displaystyle\sum\limits_{1\leq i<j\leq m}}
(ij).
\]
$T_{m}$ acts naturally on any representation of $S_{m}$. Let $\Upsilon
:\mathbb{V}\rightarrow\mathbb{V}$ be the equivariant operator associated to
$U$.

\begin{lemma}
The restriction $\Upsilon:\mathbb{W}\rightarrow\mathbb{V}$ of $\Upsilon$ to
$\mathbb{W}$ is given by the formula
\[
\Upsilon(w)=T_{n}\cdot w-T_{k}\cdot w-\binom{n-k}{2}w\text{ \ for }%
w\in\mathbb{W}\text{.}%
\]

\end{lemma}

\begin{proof}
Let $\sigma:I_{k}\rightarrow I_{n}$ be a $k$-permutation corresponding to a
permutation of $I_{k}$. If a vertex $\gamma\in V(n,k)$ shares an edge with
$\sigma$ then there is exactly one value $i\in I_{k}$ such that $\sigma
(i)\neq\gamma(i)$ and $\gamma(i)=j\in I_{n}-I_{k}$. In particular
$\gamma(i)=j$ with $k<j$ and $\gamma=(ij)\cdot\sigma$. The other type of
transpositions that appear in $T_{n}-T_{k}$ have the form $(ij)$ with
$k<i<j\leq n$. and are from the group $S\left(  I_{n}-I_{k}\right)  $. Each of
this second type fixes every vector in $\mathbb{W}$. Summing up over these two
types of transpositions gives the desires result.
\end{proof}

If $\lambda\vdash m$, we let
\[
\chi_{\lambda}:S_{m}\rightarrow\mathbb{C}%
\]
be the character associated to the irreducible representation $S^{\lambda}$.
It is well known that $\chi_{\lambda}(g)$ is an integer for each $g\in G$.

\begin{lemma}
$T_{m}$ is in the center of $\mathbb{C}\left[  S_{m}\right]  $ and acts on the
irreducible representation $S^{\lambda}$ according to the formula
\[
T_{m}\cdot v=\binom{m}{2}\frac{\chi_{\lambda}\left(  \tau\right)  }%
{\chi_{\lambda}\left(  1\right)  }v\text{ for }v\in S^{\lambda}%
\]
where $\tau$ is any transposition of $S_{m}$.
\end{lemma}

\begin{proof}
We use the fact the set $C$ of transpositions is a conjugacy class in $S_{m}$.
Since
\[
T_{m}=%
{\displaystyle\sum\limits_{g\in C}}
g
\]
it follows immediately $T_{m}$ is in the center of $\mathbb{C}\left[
S_{m}\right]  $ and therefore acts by a scalar on an irreducible
representation. If $\alpha\in\mathbb{C}$ is the eigenvalue for the action of
$T_{m}$ on $S^{\lambda}$ then taking traces yields
\[
\chi_{\lambda}\left(  1\right)  \alpha=%
{\displaystyle\sum\limits_{g\in C}}
\chi_{\lambda}(g)
\]
which is the desired result.
\end{proof}

Suppose $\lambda\vdash k$ and recall that $\mathbb{W}_{\lambda}=m_{\lambda
}S^{\lambda}$ is the corresponding isotypic component in $\mathbb{W}$.

\begin{proposition}
The restriction of $\Upsilon$ to Ind$_{H}^{G}(\mathbb{W}_{\lambda})$ is given
by the formula
\[
\Upsilon(v)=T_{n}\cdot v-\left(  \binom{k}{2}\frac{\chi_{\lambda}\left(
\tau\right)  }{\chi_{\lambda}\left(  1\right)  }+\binom{n-k}{2}\right)
v\text{\ \ for }v\in\text{Ind}_{H}^{G}(\mathbb{W}_{\lambda})\text{.}%
\]

\end{proposition}

\begin{proof}
It follows immediately from the previous two lemmas that the formula is valid
for the restriction of $\Upsilon$ to the subspace $\mathbb{W}_{\lambda}$. Now
suppose $w\in\mathbb{W}_{\lambda}$ and $g\in S_{n}$. Since $\Upsilon$ is
$S_{n}$-equivariant we have
\[
\Upsilon(g\cdot w)=g\cdot \Upsilon(w)=g\cdot\left(  T_{n}\cdot w-\left(
\binom{k}{2}\frac{\chi_{\lambda}\left(  \tau\right)  }{\chi_{\lambda}\left(
1\right)  }+\binom{n-k}{2}\right)  w\right)  =
\]%
\[
T_{n}\cdot\left(  g\cdot w\right)  -\left(  \binom{k}{2}\frac{\chi_{\lambda
}\left(  \tau\right)  }{\chi_{\lambda}\left(  1\right)  }+\binom{n-k}%
{2}\right)  g\cdot w.
\]
Hence the result follows since every vector in Ind$_{H}^{G}(\mathbb{W}%
_{\lambda})$ is a sum of vectors of the form $g\cdot w$ for $g\in S_{n}$ and
$w\in\mathbb{W}_{\lambda}$.
\end{proof}

Observe that the previous lemma implies that the space Ind$_{H}^{G}%
(\mathbb{W}_{\lambda})$ is $\Upsilon$-invariant, since it's clear that
Ind$_{H}^{G}(\mathbb{W}_{\lambda})$ is $T_{n}$-invariant. Thus, by Proposition
2.3, it is sufficient to calculate $\Upsilon$ on the subspaces Ind$_{H}%
^{G}(\mathbb{W}_{\lambda})$ for $\lambda\vdash k$. Hence, we need to
understand the action of $T_{n}$ on Ind$_{H}^{G}(\mathbb{W}_{\lambda})$. In
order to do this, by Lemma 3.2, we need to decompose Ind$_{H}^{G}%
(\mathbb{W}_{\lambda})$ into irreducible $S_{n}$-modules. It is the
Littlewood-Richardson rule \cite{james}, or rather, the simpler version
referred to as Pieri's formula, that solves this last problem. To make the
application of the rule specific, we write $\mathbb{W}_{\lambda}=m_{\lambda
}S^{\lambda}\otimes\mathbb{C}$ where $\mathbb{C}$ represents the trivial
module for $S\left(  I_{n}-I_{k}\right)  $ and the group $H=S_{k}\times$
$S\left(  I_{n}-I_{k}\right)  $ acts by the formula $(\sigma,g)\cdot v\otimes
z=\sigma\cdot v\otimes g\cdot z$. Since
\[
\text{Ind}_{H}^{G}(\mathbb{W}_{\lambda})=\text{Ind}_{H}^{G}(m_{\lambda
}S^{\lambda}\otimes\mathbb{C})\cong m_{\lambda}\text{Ind}_{H}^{G}(S^{\lambda
}\otimes\mathbb{C})
\]
it suffices to decompose Ind$_{H}^{G}(S^{\lambda}\otimes\mathbb{C})$ into
irreducible $S_{n}$-modules, which is exactly what Pieri's formula does. To
describe the application of the rule in this context, we identify a partition
$\lambda$ with the corresponding Young diagram $Y(\lambda)$. Suppose
$\lambda\vdash k$ and $\mu\vdash n$. We write $\lambda\prec\mu$ if the Young
diagram $Y(\mu)$ for $\mu$ can be obtained from the diagram $Y(\lambda)$ by
adding at most one box to each column.

\begin{theorem}
[Pieri's formula]If $\lambda\vdash k$ then
\[
\text{Ind}_{H}^{G}(S^{\lambda}\otimes\mathbb{C})\cong%
{\displaystyle\bigoplus\limits_{\lambda\prec\mu}}
S^{\mu}\text{.}%
\]

\end{theorem}

This gives us our main result:

\begin{theorem}
Suppose $\tau_{n}$ is any transposition of $S_{n}$ and $\tau_{k}$ is any
transposition of $S_{k}$. The eigenvalues of the arrangement graph $A(n,k)$
are the numbers of the form
\[
\binom{n}{2}\frac{\chi_{\mu}\left(  \tau_{n}\right)  }{\chi_{\mu}\left(
1\right)  }-\binom{k}{2}\frac{\chi_{\lambda}\left(  \tau_{k}\right)  }%
{\chi_{\lambda}\left(  1\right)  }-\binom{n-k}{2}%
\]
where $\lambda\vdash k$ and where $\mu\vdash n$ such that $\lambda\prec\mu$.
\end{theorem}

To see that these numbers are integers and to help calculate their values, one
can apply a formula originally attributed to Frobenius \cite{diaconis} and
\cite{ingram}. In particular, suppose $\lambda=(\lambda_{1},\ldots,\lambda
_{l})\vdash m$. If $\tau\in S_{m}$ is any transposition, then we have the
following:
\[
\binom{m}{2}\frac{\chi_{\lambda}\left(  \tau\right)  }{\chi_{\lambda}\left(
1\right)  }=%
{\displaystyle\sum\limits_{j=1}^{l}}
\left(  \binom{\lambda_{j}-j+1}{2}-\binom{j}{2}\right)  =\left(
{\displaystyle\sum\limits_{j=1}^{l}}
\binom{\lambda_{j}-j+1}{2}\right)  -\binom{l+1}{3}%
\]
since
\[%
{\displaystyle\sum\limits_{j=1}^{l}}
\binom{j}{2}=\binom{l+1}{3}.
\]

\section{A conjecture by Chen, Ghorbani and Wong}

In \cite[Conjecture 3]{chen1}, Chen Ghorbani and Wong conjecture that for $k$
fixed and $n$ large that $-k$ is the only negative eigenvalue in the spectrum
of the $\left(  n,k\right)  $-arrangement graph. In this section we prove
their conjecture. Let $\mathbb{Q}$ denote the rational numbers.

\begin{proposition}
There is a polynomial $p(x)\in$ $\mathbb{Q}\left[  x\right]  $ such that when
$n>p(k)$ then $-k$ is the only negative eigenvalue in the spectrum of the
$\left(  n,k\right)  $-arrangement graph.
\end{proposition}

We prove the proposition by establishing two lemmas.

\begin{lemma}
Suppose $\lambda=(\lambda_{1},\ldots,\lambda_{q})\vdash k$, $\mu=(\mu
_{1},\ldots,\mu_{r})\vdash n$, $\lambda\prec\mu$ and $\mu_{1}=n-k.$ If
$n-k>k$. Then
\[
\binom{n}{2}\frac{\chi_{\mu}\left(  \tau_{n}\right)  }{\chi_{\mu}\left(
1\right)  }-\binom{k}{2}\frac{\chi_{\lambda}\left(  \tau_{k}\right)  }%
{\chi_{\lambda}\left(  1\right)  }-\binom{n-k}{2}=-k.
\]

\end{lemma}

\begin{proof}
The relation $\lambda\prec\mu$ means we have to add $n-k$ boxes to the Young
diagram $Y(\lambda)$, no more than one to each column, to obtain the Young
diagram $Y(\mu)$. Since there are $\lambda_{1}\leq k<n-k$ columns in
$Y(\lambda)$ we must add a positive number of boxes $a$ to the first row to
obtain $\lambda_{1}+a=\mu_{1}=n-k$. But then the number of left over boxes to
be added to $Y(\lambda)$ is exactly $n-k-a=\lambda_{1}$ so we must add one box
to each column. Thus
\[
\mu_{j+1}=\lambda_{j}\text{ for }1\leq j\leq q\text{.}%
\]
Using the formula for the characters and the fact that $r=q+1,$ we obtain
\[
\binom{n}{2}\frac{\chi_{\mu}\left(  \tau_{n}\right)  }{\chi_{\mu}\left(
1\right)  }-\binom{k}{2}\frac{\chi_{\lambda}\left(  \tau_{k}\right)  }%
{\chi_{\lambda}\left(  1\right)  }=
\]%
\[
\binom{n-k}{2}-\binom{q+1}{2}+\binom{\lambda_{1}-1}{2}-\binom{\lambda_{1}}%
{2}+\cdots+\binom{\lambda_{q}-q}{2}-\binom{\lambda_{q}-q+1}{2}=
\]

\[
\binom{n-k}{2}-\binom{q+1}{2}+1-\lambda_{1}+1-(\lambda_{2}-1)+\cdots
+1-(\lambda_{q}-(q-1))=
\]%
\[
\binom{n-k}{2}-\binom{q+1}{2}+1+2+\cdots+q-(\lambda_{1}+\cdots+\lambda_{q}).
\]
Thus the result follows.
\end{proof}

\begin{lemma}
Suppose $\lambda=(\lambda_{1},\ldots,\lambda_{q})\vdash k$, $\mu=(\mu
_{1},\ldots,\mu_{r})\vdash n$, $\lambda\prec\mu$ and $\mu_{1}>n-k.$There is a
polynomial $p(x)\in\mathbb{Q}\left[  x\right]  $ such that when $n>p(k)$ then
the associated eigenvalue is positive.
\end{lemma}

\begin{proof}
Since
\[
\frac{\chi_{\lambda}\left(  \tau_{k}\right)  }{\chi_{\lambda}\left(  1\right)
}\leq1
\]
it follows that
\[
-\binom{k}{2}\frac{\chi_{\lambda}\left(  \tau_{k}\right)  }{\chi_{\lambda
}\left(  1\right)  }\geq-\binom{k}{2}.
\]
Using the character formula for the character $\chi_{\mu}$ and the fact that
$\mu_{1}>n-k$ we see that
\[
\binom{n}{2}\frac{\chi_{\mu}\left(  \tau_{n}\right)  }{\chi_{\mu}\left(
1\right)  }-\binom{n-k}{2}\geq\binom{n-k+1}{2}-\binom{n-k}{2}-%
{\displaystyle\sum\limits_{j=1}^{r}}
\binom{j}{2}\geq
\]%
\[
n-k-\binom{r+1}{3}.
\]
Now use the fact that $\lambda\prec\mu$. Since the young diagram $Y(\mu)$ is
obtained from the diagram $Y(\lambda)$ by adding at most one box to a column
it follows that $r\leq q+1\leq k+1.$ Putting this together we have
\[
\binom{n}{2}\frac{\chi_{\mu}\left(  \tau_{n}\right)  }{\chi_{\mu}\left(
1\right)  }-\binom{k}{2}\frac{\chi_{\lambda}\left(  \tau_{k}\right)  }%
{\chi_{\lambda}\left(  1\right)  }-\binom{n-k}{2}\geq n-k-\binom{k+2}%
{3}-\binom{k}{2}.
\]
This will be positive if
\[
n>\binom{k+2}{3}+\binom{k}{2}+k.
\]
Therefore we could use
\[
p(k)=\frac{1}{6}k\left(  k+1\right)  \left(  k+5\right)
\]

\end{proof}

To finish the proof of the conjecture we only need to see that the condition
$\lambda\prec\mu$ implies that $\mu_{1}\geq n-k.$ But this is clear since
$\lambda_{1}$ is the number of columns in the Young diagram $Y(\lambda)$ and
we need to add a total of $n-k$ boxes to obtain the Young diagram $Y(\mu)$.
Hence, if we add $a$ boxes to the first row to obtain $\mu_{1}=\lambda_{1}+a$
we can at add at most $\lambda_{1}$ additional boxes. Thus we need
$\lambda_{1}+a\geq n-k$.

We also make the following observation. Our proof shows that when $n>p\left(
k\right)  $, given $\lambda=(\lambda_{1},\ldots,\lambda_{q})\vdash k$ then
there is exactly one partition $\mu\left(  \lambda\right)  =\left(
n-k,\lambda_{1},\ldots,\lambda_{q}\right)  \vdash n$ such that $\lambda
\prec\mu\left(  \lambda\right)  $ and such that the formula for the
eigenvalues yields the value $-k$ (this is the partition descibed in the proof
of Lemma 4.2). Since
\[
\text{Ind}_{H}^{G}(S^{\lambda}\otimes\mathbb{C})\cong%
{\displaystyle\bigoplus\limits_{\lambda\prec\mu}}
S^{\mu}%
\]
and since then multiplicity of $S^{\lambda}$ in $\mathbb{W}$ is $\chi
_{\lambda}\left(  1\right)  $ it follows that the multiplicity of the
eigenvalue $-k$, for $n>p\left(  k\right)  ,$ is exactly
\[%
{\displaystyle\sum\limits_{\lambda\vdash k}}
\chi_{\mu(\lambda)}\left(  1\right)  \chi_{\lambda}\left(  1\right)  .
\]

\end{document}